\documentclass[11pt, leqno]{amsart}
\usepackage{amsmath,amssymb,txfonts}
\usepackage{amssymb}
\usepackage{amsxtra}
\usepackage{amsthm, color}
\usepackage{txfonts}
\usepackage{graphicx}
\usepackage{times}
\usepackage{tikz, hyperref}
\usetikzlibrary{calc}
\usepackage{pgfplots}
\usepackage{citeref}
\usepackage{pgfplots}
\usepackage{tikz-3dplot}
\numberwithin{equation}{section}

\usepackage[cp1252]{inputenc}
\usepackage{bbm}
\usepackage{mathrsfs}
\usepackage{graphicx}

\usepackage[active]{srcltx}

\usepackage{graphicx}

\newtheorem{prop}{Proposition}[section]
\newtheorem{theorem}[prop]{Theorem}
\newtheorem{lemma}[prop]{Lemma}

\usepackage{tikz}
\usetikzlibrary{calc}

\makeatletter

\def\hyper@x#1,#2\relax{#1}
\def\hyper@y#1,#2\relax{#2}
\def\hyper@coords#1{#1}

\newif\ifhyper@vertical

\def\hyper@computer#1#2{%
	\edef\hyper@toscan{(#1)}
	\tikz@scan@one@point\hyper@coords\hyper@toscan
	\edef\hyper@sx{\the\pgf@x}
	\edef\hyper@sy{\the\pgf@y}
	\edef\hyper@toscan{(#2)}
	\tikz@scan@one@point\hyper@coords\hyper@toscan
	\edef\hyper@ex{\the\pgf@x}
	\edef\hyper@ey{\the\pgf@y}
	\pgfmathsetmacro{\hyper@mx}{(\hyper@ex + \hyper@sx)/2}
	\pgfmathsetmacro{\hyper@my}{(\hyper@ey + \hyper@sy)/2}
	\pgfmathsetmacro{\hyper@dx}{\hyper@ex - \hyper@sx}
	\pgfmathparse{\hyper@dx == 0 ? "\noexpand\hyper@verticaltrue" : "\noexpand\hyper@verticalfalse"}
	\pgfmathresult
	\ifhyper@vertical
	\edef\hyper@cmd{-- (\tikztotarget)}
	\else
	\pgfmathsetmacro{\hyper@dy}{\hyper@ey - \hyper@sy}
	\pgfmathsetmacro{\hyper@t}{\hyper@my/\hyper@dx}
	\pgfmathsetmacro{\hyper@cx}{\hyper@mx + \hyper@t * \hyper@dy}
	\pgfmathsetmacro{\hyper@radius}{veclen(\hyper@cx - \hyper@sx, \hyper@sy)}
	\pgfmathsetmacro{\hyper@sangle}{180 - atan2(\hyper@sy,\hyper@cx-\hyper@sx)}
	\pgfmathsetmacro{\hyper@eangle}{180 - atan2(\hyper@ey,\hyper@cx-\hyper@ex)}
	\edef\hyper@cmd{arc[radius=\hyper@radius pt, start angle=\hyper@sangle, end angle=\hyper@eangle]}
	\fi
}

\def\hyper@disc@computer#1#2{%
	\edef\hyper@toscan{(#1)}
	\tikz@scan@one@point\hyper@coords\hyper@toscan
	\edef\hyper@sx{\the\pgf@x}
	\edef\hyper@sy{\the\pgf@y}
	\edef\hyper@toscan{(#2)}
	\tikz@scan@one@point\hyper@coords\hyper@toscan
	\edef\hyper@ex{\the\pgf@x}
	\edef\hyper@ey{\the\pgf@y}
	\pgfmathsetmacro{\hyper@det}{\hyper@sx * \hyper@ey - \hyper@sy * \hyper@ex}
	\pgfmathparse{\hyper@det == 0 ? "\noexpand\hyper@verticaltrue" : "\noexpand\hyper@verticalfalse"}
	\pgfmathresult
	\ifhyper@vertical
	\edef\hyper@cmd{-- (\tikztotarget)}
	\else
	\pgfmathsetmacro{\hyper@mx}{(\hyper@ex + \hyper@sx)/2}
	\pgfmathsetmacro{\hyper@my}{(\hyper@ey + \hyper@sy)/2}
	\pgfmathsetmacro{\hyper@dx}{\hyper@ex - \hyper@sx}
	\pgfmathsetmacro{\hyper@dy}{\hyper@ey - \hyper@sy}
	\pgfmathsetmacro{\hyper@dradius}{\pgfkeysvalueof{/tikz/hyperbolic disc radius}}
	\pgfmathsetmacro{\hyper@t}{((\hyper@dradius)^2 - \hyper@sx * \hyper@ex - \hyper@sy * \hyper@ey)/(2 * (\hyper@sx * \hyper@ey - \hyper@sy * \hyper@ex))}
	\pgfmathsetmacro{\hyper@radius}{sqrt((\hyper@t)^2 + .25) * veclen(\hyper@dx,\hyper@dy)}
	\pgfmathsetmacro{\hyper@cx}{\hyper@mx + \hyper@t * \hyper@dy}
	\pgfmathsetmacro{\hyper@cy}{\hyper@my - \hyper@t * \hyper@dx}
	\pgfmathsetmacro{\hyper@sangle}{atan2(\hyper@sy-\hyper@cy,\hyper@sx - \hyper@cx)}
	\pgfmathsetmacro{\hyper@eangle}{atan2(\hyper@ey-\hyper@cy,\hyper@ex - \hyper@cx)}
	\pgfmathsetmacro{\hyper@eangle}{\hyper@eangle > \hyper@sangle + 180 ? \hyper@eangle - 360 : \hyper@eangle}
	\edef\hyper@cmd{arc[radius=\hyper@radius pt, start angle=\hyper@sangle, end angle=\hyper@eangle]}
	\fi
}

\def\hyper@plane@tangent#1#2{%
	\edef\hyper@toscan{(#1)}
	\tikz@scan@one@point\hyper@coords\hyper@toscan
	\edef\hyper@sx{\the\pgf@x}
	\edef\hyper@sy{\the\pgf@y}
	\edef\hyper@toscan{(#2)}
	\tikz@scan@one@point\hyper@coords\hyper@toscan
	\edef\hyper@ex{\the\pgf@x}
	\edef\hyper@ey{\the\pgf@y}
	\pgfmathsetmacro{\hyper@ex}{\hyper@ex - \hyper@sx}
	\pgfmathsetmacro{\hyper@ey}{\hyper@ey - \hyper@sy}
	\pgfmathparse{\hyper@ex == 0 ? "\noexpand\hyper@verticaltrue" : "\noexpand\hyper@verticalfalse"}
	\pgfmathresult
	\ifhyper@vertical
	\pgfmathsetmacro{\hyper@d}{\hyper@ey/1cm}
	\pgfmathsetmacro{\hyper@radius}{\hyper@sy * exp(\hyper@d) - \hyper@sy}
	\edef\hyper@cmd{-- ++(0,\hyper@radius pt)}
	\else
	\pgfmathsetmacro{\hyper@d}{\hyper@ex > 0 ? veclen(\hyper@ex,\hyper@ey) : -veclen(\hyper@ex,\hyper@ey)}
	\pgfmathsetmacro{\hyper@radius}{abs(\hyper@sy * \hyper@d / \hyper@ex)}
	\pgfmathsetmacro{\hyper@sangle}{90 + atan(\hyper@ey/\hyper@ex)}
	\pgfkeysgetvalue{/tikz/hyperbolic plane target angle}{\hyper@eangle}
	\ifx\hyper@eangle\pgfutil@empty
	\pgfmathsetmacro{\hyper@d}{\hyper@d/1cm}
	\pgfmathsetmacro{\hyper@ey}{\hyper@ey/1cm}
	\pgfmathsetmacro{\hyper@tanhd}{tanh(\hyper@d)}
	\pgfmathsetmacro{\hyper@eangle}{acos((\hyper@d * \hyper@tanhd - \hyper@ey)/(\hyper@d - \hyper@ey * \hyper@tanhd))}
	\fi
	\edef\hyper@cmd{arc[radius=\hyper@radius pt, start angle=\hyper@sangle, end angle=\hyper@eangle]}
	\fi
}

\tikzset{%
	hyperbolic disc radius/.initial={1cm},
	hyperbolic plane/.style={
		to path={
			\pgfextra{\hyper@computer\tikztostart\tikztotarget}
			\hyper@cmd
		}
	},
	hyperbolic plane tangent/.style={
		to path={
			\pgfextra{\hyper@plane@tangent\tikztostart\tikztotarget}
			\hyper@cmd
		}
	},
	hyperbolic disc/.style={
		to path={
			\pgfextra{\hyper@disc@computer\tikztostart\tikztotarget}
			\hyper@cmd
		}
	},
	hyperbolic plane target angle/.initial={},
}

\makeatother

\begin{document}

	\title[Heat dispersion laws in smooth compact manifolds]{Heat dispersion laws in smooth compact manifolds}
	
	\author{\tiny Xiaoshang Jin}
	\address{School of Mathematics and Statistics, Huazhong University of Science and Technology, Wuhan, Hubei 430074, China}
	\email{jinxs@hust.edu.cn}
	\author{\tiny Jie Xiao}
	\address{Department of Mathematics \& Statistics,
		Memorial University, St. John's, NL A1C 5S7, Canada}
	\email{jxiao@mun.ca}


	\begin{abstract}
		Given a Lipschitz conductor $K$ in the smooth compact Riemannian $2\le n$-manifold $(M,g)$, such a half generic heat dispersion law
		$$
	{\rm H^d}_{p,\varPhi,\varPsi}(K,M)=2^{-1}	{\rm H^d}_{\Delta_p,\varPhi,\varPsi}(K,M)
$$
is not only newly-established via Theorem \ref{t11} but also deeply-explored through not only Proposition \ref{p32} (a comparison law for the generic heat dispersion) but also Proposition \ref{p31} (a recycling law for the quasilinear Laplace-Robin eigenvalue).

	\end{abstract}

	\thanks{The first-named \& second-named authors were supported by {``the Fundamental Research Funds for the Central Universities'', HUST: \# 2025BRSXB002.} \& NSERC of Canada \# 202979 respectively.}
	
	\date{}

\date{}

%\dedicatory{In memory of Adriano M. Garsia 1928-2010}

%\dedicatory{Dedicated to Adriano M. Garsia who surely appreciated a simplified approach}


\maketitle

\tableofcontents

\section{Introduction}\label{s1}
\setcounter{equation}{0}

First of all, acccording to \cite[Theorem 5.3]{Br2} (cf. \cite[Theorem 1-Corollary 2-Theorem 3]{Br1} \& \cite{Br3}), if $\Sigma$ is a compact $2\le n$-dimensional hypersurface in the $(n+1)$-dimensional Euclidean space $\mathbb R^{n+1}$, possibly with boundary $\partial \Sigma$, then
\begin{equation}
\label{e10}
\begin{cases}
 n|\mathbb B^{n}|^\frac1n\left(\int_{\Sigma}f^\frac{n}{n-1}\right)^\frac{n-1}{n}\le\int_{\Sigma}\Big(|\nabla^\Sigma f|^2+|\mathsf{H}|^2f^2\Big)^\frac12+\int_{\partial \Sigma}f\\
\forall\ \text{smooth $f>0$ on $\Sigma$ with its surface gradient}\\
\nabla^\Sigma f=\nabla f-(\nabla f\cdot{\bf n})\bf{n},
\end{cases}
\end{equation}
with \eqref{e10}'s equality holding iff not only $f$ is a constant but aslo $\Sigma$ is a flat disk,
where not only $\mathsf{H}$ is the mean curvature of $\Sigma$ but also $|\mathbb B^n|$ denotes the volume of the open unit ball $\mathbb B^n$ in $\mathbb R^n$ as well as ${\bf n}$ is the unit outer normal to $\Sigma$ with
$$
\nabla^\Sigma\cdot{\bf n}=\nabla\cdot{\bf n}=\mathsf{H}.
$$
Recall that within the study of a capillary surface which represents the interface between two different fluids, the gradient of spatially varying surface tension makes no much sense, yet the surface gradient $\nabla^\Sigma$ does via serving certain purposes from mathematical fluid mechanics. So, it is worth mentioning two special assertions.
\begin{itemize}
	\item \eqref{e10}, plus the elementry inequality
	$$
	(a^2+b^2)^\frac12\le a+b\ \ \forall\ \ (a,b)\in [0,\infty)^2,
	$$
	derives
\begin{equation}
\label{e100}
n|\mathbb B^{n}|^\frac1n\left(\int_{\Sigma}f^\frac{n}{n-1}\right)^\frac{n-1}{n}\le \int_{\Sigma}\Big(|\nabla^\Sigma f|+|\mathsf{H}|f\Big)\,+\int_{\partial \Sigma}f\ \
\forall\ \ \text{smooth $f>0$ on $\Sigma$},
\end{equation}
whence letting
$$
f=h^p\ \ \&\ \ p\in [1,\infty)\ \ \text{within}\ \ \eqref{e100},
$$
along with the H\"older inequality, produces
\begin{equation}
\label{e1001}
n|\mathbb B^{n}|^\frac1n\left(\int_{\Sigma}h^\frac{pn}{n-1}\,\right)^\frac{n-1}{n}\le\int_{\Sigma}h^p\Bigg(|\mathsf{H}|+p\left(\frac{\int_{\Sigma}|\nabla^\Sigma h|^p}{\int_{\Sigma}h^p}\right)^\frac1p\Bigg)\,+\int_{\partial \Sigma}h^p\
\forall\ \text{smooth $h>0$ on $\Sigma$}.
\end{equation}
\item Geometrically speaking, $h=1$ in \eqref{e1001} deduces the isoperimetry (cf. \cite[Corollary 2]{BE})
\begin{equation}
\label{e1000}
 n|\mathbb B^{n}|^\frac1n|\Sigma|^\frac{n-1}{n}\le|{\partial \Sigma}|+\int_{\Sigma}|\mathsf{H}|.
\end{equation}
As described within \cite[Corollary 5.4]{Br2}, if $\Sigma$ is a compact minimal hypersurface in $\mathbb R^{n+1}$ -i.e.- $\mathsf{H}=0$ with boundary $\partial\Sigma$, then
$$
n|\mathbb B^{n}|^\frac1n|\Sigma|^\frac{n-1}{n}\le|\partial\Sigma|\ \ \text{with equality holding iff $\Sigma$ is a flat disk}.
$$
\end{itemize}

Next, a combination of not only \eqref{e100}-\eqref{e1001}-\eqref{e1000} but also \cite{ChW, DFN, JX, XX} leads to a case study of the generic heat dispersions of any smooth compact Riemannian $2\le n$-manifold $(M,g)$ with boundary $\partial M$ - volume $\upsilon_g$ - surface-area $\sigma_g$:
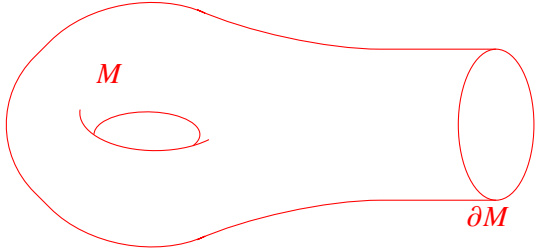
\begin{center}
\begin{tikzpicture}{\color{red}
\draw[rounded corners=36pt](7,-1)--(4.2,-1)--(2,-2)--(0,0) -- (2,2)--(4.2,1)--(7,1);
\draw (1.5,0.2) arc (175:315:1cm and 0.5cm);
\draw (2.99,-0.28) arc (-30:180:0.7cm and 0.3cm);
\draw (7.5,0) arc (0:360:0.5cm and 1cm);
\node (a) at (20:2) {$M$};
\node (a) at (-10:7) {$\partial M$};}
\end{tikzpicture}
\end{center}
%\begin{tikzpicture}
%\draw[rounded corners=35pt](7.5,-1)--(4.2,-1)--(2,-2)--(0,0) -- (2,2)--(4.2,1)--(7.5,1);
%\draw (1.5,0.2) arc (175:315:1cm and 0.5cm);
%\draw (3,-0.28) arc (-30:180:0.7cm and 0.3cm);
%\draw (6,0) arc (0:360:0.5cm and 1cm);
%\draw (8,0) arc (0:360:0.5cm and 1cm);
%\node (a) at (20:2.5) {$M$};
%\node (a) at (12:6.5) {$[0,1]$};
%\node (a) at (-12:7.5) {$\partial M$};
%\end{tikzpicture}
More precisely, given a power $p\in [1,\infty)$, a nonnegative smooth function pair $\big\{\varPhi,\varPsi\big\}$ on the pair $\{M,\partial M\}$, and a compact connected set $K\subseteq M$ which not only exists as a conductor of the unit temperature but also is thermally insulated by surrounding it with a layer of thermal insulator $M\setminus K$ with $K\subseteq M$  \& Lipschitz bounary pair $\{\partial K,\partial\Omega\}$, the so-called generic heat dispersion of a given condenser $(K,M)$ is defined as (cf. \cite{AC, Ba, DNT} for the Euclidean case)
\begin{equation}
\label{e11}
{\rm H^d}_{p,\varPhi,\varPsi}(K, M)=\underset{f\in W^{1,p}(M)\ \&\ f\big|_K=1}{\inf}\left(\int_{M}\Big(|\nabla f|^p+\varPhi|f|^p\Big)\,d\upsilon_g+\int_{\partial M}\varPsi|f|^p\,d\sigma_g\right),
\end{equation}
where
$$
\begin{cases}
W^{1,p}( M)=\Big\{f:\ \ \|f\|_{1,p, M}^p=\int_ M\Big(|\nabla f|^p+|f|^p\Big)\,d\upsilon_g<\infty\Big\};\\
W^{1,p}_0( M)=\text{the closure of $C_0^\infty( M)$ under the norm $\|\cdot\|_{1,p, M}$};\\
C_0^\infty( M)=\Big\{\text{all infinitely differentiable functions with compact support in $ M$}\Big\}.
\end{cases}
$$
Quite importantly, because we are here seeking the infimum of an energy, we can always assume that each function $f$ satisfies that $0\leq f\leq 1$ - otherwise - we can consider $-f$ or $2-f$ instead. Below are two special circumstances.
\begin{itemize}
	\item Since
	$$
	f\in W^{1,p}_0( M)\subseteq W^{1,p}( M) \Longrightarrow f\big|_{\partial M}=0,
	$$
	one has
	\begin{equation}
	\label{e12}
	{\rm H^d}_{p,\varPhi,\varPsi}(K, M)\le\underset{f\in W^{1,p}_0( M)\ \&\ f\big|_K=1}{\inf}\int_{ M}\Big(|\nabla f|^p+\varPhi|f|^p\Big)\,d\upsilon_g\ \forall\ p\in [1,\infty).
	\end{equation}

	\item Since
	$$
	1\in W^{1,p}( M)\Longrightarrow |\nabla 1|=0\ \ \forall\ \ p\in [1,\infty),
	$$
	one has
	\begin{equation}
	\label{e14}
	{\rm H^d}_{p,\varPhi,\varPsi}(K, M)\le\int_{ M}\varPhi\,d\upsilon_g +\int_{\partial M}\varPsi\,d\sigma_g\ \ \forall\ \ p\in [1,\infty).
	\end{equation}
	Of geometric interest is that if $K=M$ within \eqref{e14} then
	$$
	|\nabla f|\big|_ M=0=(f-1)\big|_{\partial M},
	$$
	whence \eqref{e14} reaches its equality (cf. \eqref{e1000}'s left-hand-side)
	\begin{equation}
	\label{e14a}
	{\rm H^d}_{p,\varPhi,\varPsi}(M, M)=\int_{ M}\varPhi\,d\upsilon_g +\int_{\partial M}\varPsi\,d\sigma_g\ \ \forall\ \ p\in [1,\infty).
	\end{equation}
	Especially, if $\varPhi=1=\varPsi$ in \eqref{e14a}, then
\begin{equation*}
{\rm H^d}_{p,1,1}(M, M)=\upsilon_g(M)+\sigma_g({\partial M})\ \ \forall\ \ p\in [1,\infty).
\end{equation*}
\end{itemize}

Finally, with the foregoing concepts we get the following brand-new principle of mathematical-physical interest.

\begin{theorem}
	\label{t11} Let
	$$
	\begin{cases}
	1<p<\infty;\\
	\Big(0\leq \varPhi, 0\le\varPsi\Big)\in C^\infty(M)\times C^\infty(\partial M);\\
	\Delta_p f=\text{div}(|\nabla f|^{p-2}\nabla f);\\
	\mathscr{F}(K, M)=\left\{f\in C^2(M): \begin{cases} 0\leq f\leq 1;\\
	f=1\ \text{on}\ K;\\
	\frac{{\nabla f}\cdot{\bf n}}{|\nabla f|^{2-p}}+\frac{\varPsi f}{|f|^{2-p}}=0\  \text{on}\  \partial M.
	\end{cases}\right\};\\
	{\rm H^d}_{\Delta_p,\varPhi,\varPsi}(K, M)
	=\underset{f\in\mathscr{F}(K, M)}{\inf}\left(\int_{M}\bigg(\big|\Delta_p f-{\varPhi}{f^{p-1}}\big|+{\varPhi}{f^{p-1}}\bigg)\,d\upsilon_g+\int_{\partial M}\frac{\varPsi }{f^{1-p}}\,d\sigma_g\right).
	
		\end{cases}
		$$
		Then there exists the half generic heat dispersion law
		\begin{equation}
		\label{e18}
	  {\rm H^d}_{p,\varPhi,\varPsi}(K, M)=2^{-1}	{\rm H^d}_{\Delta_p,\varPhi,\varPsi}(K, M).
		\end{equation}
			Consequently, there are two extreme situations.
			\begin{itemize}
				\item[\rm (i)] If not only $\varPhi=0$ but also $\varPsi=$ a constant $\to \infty$, then
			$$
			\begin{cases}
			\underset{\varPsi\to\infty}{\lim}{\rm H^d}_{p,0,\varPsi}(K, M)=\underset{f\in C^1_0(M)\ \&\ f\big|_{K}=1}{\inf}\int_{M}|\nabla f|^p\,d\upsilon_g;\\
			
			\underset{\varPsi\to\infty}{\lim}{\rm H^d}_{\Delta_p,0,\varPsi}(K, M)=\underset{
				f\in C^2_0(M),\ \& \ f\big|_K=1}{\inf}\int_{M}|\Delta_p f|\,d\upsilon_g.
			\end{cases}
			$$
		
		\item[\rm(ii)] If not only $\varPhi=$ a constant $\to\infty$ but also $\varPsi=0$, then
		$$
		\begin{cases}
		\underset{\varPhi\to\infty}{\lim}{\rm H^d}_{p,\varPhi,0}(K, M)=\infty;\\
		\underset{\varPhi\to\infty}{\lim}{\rm H^d}_{\Delta_p,\varPhi,0}(K, M)=\infty.
		\end{cases}
		$$
		\end{itemize}
\end{theorem}

Within \S\ref{s2}, we present a proof of Theorem \ref{t11}. Furthermore, within \S\ref{s3}, we widely explore Theorem \ref{t11} to interestingly establish not only Proposition \ref{p32} which exists as a comparison law for the generic heat dispersion but also Proposition \ref{p31} which exists as a recycling law for the quasilinear Lapace-Robin eigenvalue.

\section{Proof of Theorem \ref{t11}}\label{s2}
\setcounter{equation}{0}

\begin{lemma}\label{l21} There is a unique minimizer $f_\ast\in W^{1,p}(M)$ such that not only
\begin{equation}
\label{e14aaa}\begin{aligned}
{\rm H^d}_{p,\varPhi,\varPsi}(K, M)&=\int_{M}\Big(|\nabla f_\ast|^p+\varPhi f_\ast ^p\Big)\,d\upsilon_g+ \int_{\partial M}\varPsi f_\ast ^p\,d\sigma_g
\\ &=\int_{M}\varPhi f_\ast^{p-1}\,d\upsilon_g+\int_{\partial M} \varPsi f_\ast^{p-1}\,d\sigma_g
\end{aligned}
\end{equation}
but also
\begin{equation}\label{e14aaaa}
\begin{cases}
-\Delta_p f_\ast +\varPhi|f_\ast|^{p-2}f_\ast=0&\ \text{in\ \ $ M\setminus K$};\\
f_\ast=1&\ \text{in\ \ $K$};\\
|\nabla f_\ast|^{p-2}{\nabla f_\ast\cdot{\bf n}}+\varPsi|f_\ast|^{p-2}f_\ast=0&\ \text{on}\, \ \partial M,
\end{cases}
\end{equation}
holds in the weak sense of
\begin{equation}
\label{e14aaaaa}
\int_{ M}\Bigg(\bigg(\frac{\nabla f_\ast\cdot\nabla \phi}{|\nabla f_\ast|^{2-p}}\bigg)+\frac{\varPhi f_\ast\phi}{|f_\ast|^{2-p}}\Bigg)\,d\upsilon_g+ \int_{\partial M}\left(\frac{\varPsi f_\ast\phi}{|f_\ast|^{2-p}}\right)\,d\sigma_g=0\ \
\end{equation}
for any $\phi\in C^\infty( M)$ satisfying $\phi\big|_K=0.$
Consequently, there is the vanishing heat dispersion equivalence
$$
{\rm H^d}_{p,\varPhi,\varPsi}(K, M)=0\Longleftrightarrow\varPhi=0=\varPsi.
$$
\end{lemma}
\begin{proof}
Referring to the argument for either \cite[Lemma 4.1]{He} or \cite[Proposition 3.1]{Ba}, suppose that $\{f_j\}$ is a minimizing sequence of \eqref{e11}. Then
\begin{equation}
\label{e21e}
\begin{cases}
f_j\in W^{1,p}( M);\\
f_j\big|_K=1;\\
\underset{j\to\infty}{\lim}
\left(\int_{ M}\Big(|\nabla f_j|^p+\varPhi|f_j|^p\Big)\,d\upsilon_g+\int_{\partial M}\varPsi|f_j|^p\,d\sigma_g\right)={\rm H^d}_{p,\varPhi,\varPsi}(K, M).
\end{cases}
\end{equation}
 It follows from \eqref{e21e}'s third equality and $0\leq f_j\leq 1$ that $\{f_j\}$ is a bounded sequence in $W^{1,p}( M)$ which is reflexiv. Consequently, the Rellich-Kondrakov theorem, together with the Sobolev trace compact embedding
$$
W^{1,p}( M)\hookrightarrow L^p(\partial M),
$$
ensures a subsequence $\{f_{j_k}\}$ as well as a function $f_\ast$ such that
$$
\begin{cases}
f_\ast\in W^{1,p}( M);\\
	f_{j_k}\rightharpoonup {f_\ast}\ \ &\text{in}\ \ W^{1,p}( M);\\
	\|f_{j_k}-{f_\ast}\|^p_{p, M}=\int_{ M}|f_{j_k}-{f_\ast}|^p\,d\upsilon_g\to 0;\\
	\|f_{j_k}-{f_\ast}\|^p_{p,\partial M}=\int_{\partial M}|f_{j_k}-{f_\ast}|^p\,d\sigma_g\to 0;\\
	f_{j_k}\to {f_\ast}\ \ &\text{a.e. on}\ \  M;\\
	f_\ast\big|_K=1.
\end{cases}
$$
Clearly, the above weak convergence derives the semicontinuity
\begin{align*}
&\int_{ M}\Big(|\nabla f_\ast|^p+\varPhi|f_\ast|^p\Big)\,d\upsilon_g+\int_{\partial M}\varPsi|f_{\ast}|^p\,d\sigma_g\\
&\quad\le\underset{k\to\infty}{\liminf}\left(\int_{ M}\Big(|\nabla f_{j_k}|^p+\varPhi|f_{j_k}|^p\Big)\,d\upsilon_g+\int_{\partial M}\varPsi|f_{j_k}|^p\,d\sigma_g\right).
\end{align*}
This implies the first equality of \eqref{e14aaa}. Of course, the uniqueness of $f_\ast$ follows from the fact that $|\cdot|^p$ is a strict convex function as $p\in (1,\infty)$.

Since $f_\ast$ is a unique minimizer for ${\rm H^d}_{p,\varPhi,\varPsi}(K, M)$, the induced Euler equation for $f_\ast$ leads to \eqref{e14aaaaa} -equivalently- \eqref{e14aaaa}. Finally, by setting $\phi=f_\ast-1$ in \eqref{e14aaaaa}, we obtain the second equality of \eqref{e14aaa}.

\end{proof}

Thanks to not only Lemma \ref{l21} but also the essential fact that any Lipschitz domain can be approximated by a sequence of the smooth domains within $(M,g)$, it is enough to validate \eqref{e18} for any smooth condenser $(K, M)$ according to the following three phases.

\begin{itemize}
 \item Firstly, it is easy to verify the following two identifications:
	$$
	\begin{cases}
	{\rm H^d}_{p,\varPhi,\varPsi}(K, M) =\int_K \varPhi\,d\upsilon_g\\
	\quad\quad\quad\quad\quad\quad\ \ \  +\underset{f\in \mathfrak{C}^{0,1}(K, M)}{\inf}\left(\int_{ M\setminus K}\Big(|\nabla f|^p+\varPhi f ^p\Big)\,d\upsilon_g+\int_{\partial M}\varPsi f^{p}\,d\sigma_g\right);\\
	{\rm H^d}_{\Delta_p,\varPhi,\varPsi}(K, M) =2\int_K \varPhi\,d\upsilon_g\\
		\quad\quad\quad\quad\quad\quad\ \ \ \  +\underset{f\in \mathfrak{C}^{1,1}_{\varPsi,p}(K, M)}{\inf}\left(\int_{ M\setminus K}\Big(|\Delta_p f-{\varPhi}{f^{p-1}}|+{\varPhi}{f^{p-1}}\Big)\,d\upsilon_g+\int_{\partial M}{\varPsi}{f^{p-1}} d\sigma_g\right),
	\end{cases}
	$$
	where
	$$
	\begin{cases} \mathfrak{C}^{0,1}(K, M)=\Big\{f\in C^{0,1}(\overline{ M\setminus K}):\ 0\leq f\leq 1=f|_{\partial K}\Big\};\\
	\mathfrak{C}^{1,1}_{\varPsi,p}(K, M)=\Bigg\{f\in C^{1,1}(\overline{ M\setminus K}):\ \begin{cases}0\leq f\leq 1= f\big|_{\partial K};\\
	|\nabla f|\big|_{\partial K}=\bigg(\frac{{\nabla f}\cdot{\bf n}}{|\nabla f|^{2-p}}+{\varPsi}{f^{p-1}}\bigg)\Bigg|_{\partial M}=0.
	\end{cases}
	\Bigg\}.
	\end{cases}
	$$
\item Secondly, in the sequel we prove \eqref{e18} according to two circumstances.
	\begin{itemize}
		\item On the one hand, if
		$$
		f\in \mathfrak{C}^{1,1}_{\varPsi,p}(K, M)\subseteq\mathfrak{C}^{0,1}(K, M),
		$$
		then {
		$$
		\begin{aligned}
		 &\int_{ M\setminus K}\left|\Delta_p f-{\varPhi}{f^{p-1}}\right|\,d\upsilon_g\\
		 &\quad\geq \int_{ M\setminus K}(1-2f) (\Delta_p f-{\varPhi}{f^{p-1}})\,d\upsilon_g
		\\ &\quad=\int_{\partial( M\setminus K)}\left(\frac{(1-2f){\nabla f}\cdot{\bf n}}{|\nabla f|^{2-p}}\right)\,d\upsilon_g\\
		&\quad\quad-\int_{ M\setminus K}\left(\bigg(\frac{\big(\nabla(1-2f)\big)\cdot(\nabla f)}{|\nabla f|^{2-p}}\bigg)+\bigg(\frac{(1-2f)\varPhi}{f^{1-p}}\bigg)\right)\,d\upsilon_g
		\\ &\quad =\int_{\partial  M}\bigg(\frac{(2f-1)\varPsi}{f^{1-p}}\bigg)\,d\upsilon_g+2\int_{ M\setminus K}|\nabla f|^p\,d\upsilon_g+\int_{ M\setminus K}\bigg(\frac{(2f-1)\varPhi}{f^{1-p}}\bigg)\,d\upsilon_g
		\\ &\quad =2\left(\int_{ M\setminus K}\Big(|\nabla f|^p+\varPhi f^p\Big)\,d\upsilon_g+\int_{\partial  M} \varPsi f^{p}\,d\sigma_g\right)\\
		&\quad\quad-\left(\int_{ M\setminus K}{\varPhi}{f^{p-1}}\,d\upsilon_g+\int_{\partial M}{\varPsi}{f^{p-1}}\,d\sigma_g\right).
		\end{aligned}
		$$
}
		This in turn derives
		\begin{equation}
		\label{e31e}
		{\rm H^d}_{\Delta_p,\varPhi,\varPsi}(K, M) \geq 2{\rm H^d}_{p,\varPhi,\varPsi}(K, M).
		\end{equation}
		\item On the other hand, let $0\le f_\ast\le 1$ be the minimizer as defined in Lemma \ref{l21}. Notice that
		$$\Delta_pf_\ast\ \ \&\ \ f_\ast\in C^{1,\alpha}(\overline{M\setminus K})\ \ \forall\ \ \alpha\in (0,1)$$
		by the elliptic regularity theory. Thus this last $C^{1,\alpha}$-regularity allows us to select $f_\ast$ as a test function to calculate $${\rm H^d}_{\Delta_p,\varPhi,\varPsi}(K, M)\ \ \text{as long as}\ \  \nabla f_\ast\big|_{\partial K}=0.
		$$ However, this last vanishing condition is usually not valid. To fix this issue,
	in the sequel, for sufficiently small $\delta>0,$ we are going to construct a function $h\in C^2[0,1]$ satisfying
\begin{equation}\label{e22}
		\begin{cases}
		h(t)\big|_{[0,1-\delta]}=t;\\
		h'(1-\delta)-1=	h''(1-\delta)=h'(1)=h''(1)=0;\\
		h(1)=1;\\
        \dot{h}(t)=h'(t)\geq 0\ \ \forall\ \ t\in[0,1].
		\end{cases}
\end{equation}
Upon setting
$$
 w=h(f_\ast){\Longrightarrow |\nabla w|\big|_{\partial K}=0},
		$$
		we have not only
			\begin{equation}\label{e23}\begin{aligned}
		\Delta_pw-\varPhi w^{p-1}&=\dot{h}^{p-1}\Delta_p f_\ast+(\dot{h}^{p-1})'|\nabla f_\ast|^p-\varPhi w^{p-1}
\\ & =(\dot{h}^{p-1})'|\nabla f_\ast|^p+\dot{h}^{p-1}\varPhi f_\ast^{p-1}-\varPhi h^{p-1}(f_\ast)\ \ \text{in}\ \  M\setminus K,
\end{aligned}
		\end{equation}
but also
		\begin{equation}
		\label{e23b}
		\begin{aligned}
		&\int_{ M\setminus K}\big|(\dot{h}^{p-1})'(f_\ast)\big| |\nabla f_\ast|^p\,d\upsilon_g\\
		&\quad = \int_{\{1-\delta<f_\ast<1\}} \big|(\dot{h}^{p-1})'(f_\ast)\big| |\nabla f_\ast|^p\,d\upsilon_g
		\\ &\quad = \int_{1-\delta}^1\big|(\dot{h}^{p-1})'(t)\big|\left(\int_{\{f_\ast=t\}}|\nabla f_\ast|^{p-1}\,d\sigma_g\right)\,dt,
		\end{aligned}
		\end{equation}
		whence making a two-fold treatment.
		\begin{itemize}
		
		\item On the one hand, letting $t$ be sufficiently close to $1$, along with \eqref{e14aaaa}, ensures
		\begin{equation}\label{e36}
	\begin{aligned}
		\int_{\{f_\ast=t\}}|\nabla f_\ast|^{p-1}\,d\sigma_g & =\int_{ M\cap\{f_\ast<t\}}\Delta_p f_\ast\,d\upsilon_g-\int_{\partial M}|\nabla f_\ast|^{p-2}\left(\nabla f_\ast\cdot{\bf n}\right)\,d\sigma_g
		\\& =\int_{ M\cap\{f_\ast<t\}}\varPhi f_\ast^{p-1}\,d\upsilon_g  +\int_{\partial M} \varPsi f_\ast^{p-1} d\sigma_g
\\ &\to  \int_{ M\setminus K}\varPhi f_\ast^{p-1}\,d\upsilon_g  +\int_{\partial M} \varPsi f_\ast^{p-1} d\sigma_g\ \ \Big(\text{as}\ \ \delta\to 0\Big)
\\ &={\rm H^d}_{p,\varPhi,\varPsi}(K, M){-\int_K\varPhi\,d\upsilon_g}.
		\end{aligned}
		\end{equation}
		\item On the other hand, we can select a special function $h$ such that
		\begin{equation}
		\label{e37}
		\int_{1-\delta}^1\big|(\dot{h}^{p-1})'(t)\big|\,dt\to 1\ \ \text{as}\ \ \delta\to 0.
		\end{equation}
		As a matter of fact, given a sufficiently small $\varepsilon>0,$ let
		$$
		\begin{cases} \delta=\frac{2\varepsilon^2\pi}{1-2\varepsilon}+\varepsilon\pi;\\
		h''(t)=\begin{cases}
		0&\forall\ \ t\in \Big[0,1- \frac{2\varepsilon^2\pi}{1-2\varepsilon}-\varepsilon\pi\Big); \\
		\sin\frac{1}{\varepsilon}\Big(t-1+\frac{2\varepsilon^2\pi}{1-2\varepsilon}+\varepsilon\pi\Big)&\forall\ \ t\in \Big[1- \frac{2\varepsilon^2\pi}{1-2\varepsilon}-\varepsilon\pi,1-\frac{2\varepsilon^2\pi}{1-2\varepsilon}\Big);\\ \left(\frac{1+2\varepsilon}{2}\right)\left(\frac{2\varepsilon-1}{2\varepsilon^2}\right)\sin\frac{1-2\varepsilon}{2\varepsilon^2}
		\Big(t-1+\frac{2\varepsilon^2\pi}{1-2\varepsilon}\Big) &\forall\ \ t\in \Big[1-\frac{2\varepsilon^2\pi}{1-2\varepsilon},1\Big].
		\end{cases}
		\end{cases}
		$$
		Then
		$$\begin{aligned}
		h'(t)&=1+\int_0^t h''(s)ds
		\\&=\begin{cases}
		1&\forall\ \  t\in \Big[0,1- \frac{2\varepsilon^2\pi}{1-2\varepsilon}-\varepsilon\pi\Big); \\
		1+\varepsilon\Big(1-\cos\big(\frac{t-1}{\varepsilon}+\frac{2\varepsilon\pi}{1-2\varepsilon}+\pi\big)\Big)&\forall\ \ t\in \Big[1- \frac{2\varepsilon^2\pi}{1-2\varepsilon}-\varepsilon\pi,1-\frac{2\varepsilon^2\pi}{1-2\varepsilon}\Big);\\
		1+2\varepsilon+\left(\frac{1+2\varepsilon}{2}\right)\Bigg(-1+\cos\Big(
		\frac{t-1+\frac{2\varepsilon^2\pi}{1-2\varepsilon}}{\big(\frac{1-2\varepsilon}{2\varepsilon^2}\big)^{-1}}\Big)\Bigg)& \forall\ \ t\in\Big[1-\frac{2\varepsilon^2\pi}{1-2\varepsilon},1\Big].
		\end{cases}
		\end{aligned}
		$$ This in turn implies that
		 $$h(t)=\int_0^t h'(s)ds\ \ \text{satisfies}\ \
		\eqref{e22}.
		$$
		Accordingly, there holds the required limiting process \eqref{e37}:
		$$\begin{aligned}
		\int_{1-\delta}^1\big|(\dot{h}^{p-1})'(t)\big|\,dt &=(\dot{h}^{p-1})(t)\big|^{1-\frac{2\varepsilon^2\pi}{1-2\varepsilon}}_{1- \frac{2\varepsilon^2\pi}{1-2\varepsilon}-\varepsilon\pi}+(\dot{h}^{p-1})(t)\big|^{1-\frac{2\varepsilon^2\pi}{1-2\varepsilon}}_{1}
		\\ &=2(1+2\varepsilon)^{p-1}-1
		\\ &\to 1 \ \ \text{as}\ \ \varepsilon \to 0\ \ \text{or}\ \ \delta\to 0.
		\end{aligned}
		$$
		\end{itemize}
		Now, a combination of \eqref{e23}-\eqref{e23b}-\eqref{e36}-\eqref{e37} deduces
		\begin{equation}
		\label{eAe}
		\int_{ M\setminus K}\big|(\dot{h}^{p-1})'(f_\ast)\big| |\nabla f_\ast|^p\,d\upsilon_g\to {\rm H^d}_{p,\varPhi,\varPsi}(K, M){-\int_K\varPhi\,d\upsilon_g}\ \ \text{as}\ \ \delta\to 0.
		\end{equation}
		Moreover, if $f_\ast\equiv 1$, then $\varPhi=0=\varPsi$ which ensures \eqref{e18}. Otherwise, we can assume that $0\le f_\ast<1$ in $M\setminus K$ since $f_\ast$ is $p$-superharmonic, thereby getting
\begin{equation}
\label{eBe}
\begin{aligned}
&\int_{M\setminus K}\Big|\dot{h}^{p-1}\varPhi f_\ast^{p-1}-\varPhi h^{p-1}(f_\ast)\Big|\,d\upsilon_g\\
&\ \ =\int_{\{f_\ast>1-\delta\}}\Big|\dot{h}^{p-1}\varPhi f_\ast^{p-1}-\varPhi h^{p-1}(f_\ast)\Big|\,d\upsilon_g\\
&\ \ \leq (2+2\varepsilon)\big(\max\varPhi\big){\upsilon_g}\big(\{f_\ast>1-\delta\}\big) \\ &\ \ \ {\to 0\ \ \text{as}\ \ \delta\to 0}
\end{aligned}.
\end{equation}
Consequently, \eqref{e23}, along with \eqref{eAe}-\eqref{eBe}, gives not only
$$\begin{aligned}
&\int_{ M\setminus K}|\Delta_pw-\varPhi w^{p-1}|\,d\upsilon_g\\
&\quad \leq \int_{ M\setminus K}\big|(\dot{h}^{p-1})'(f_\ast)\big| |\nabla f_\ast|^p\,d\upsilon_g +\int_{M\setminus K}\Big|\dot{h}^{p-1}\varPhi f_\ast^{p-1}-\varPhi h^{p-1}(f_\ast)\Big|\,d\upsilon_g
\\ & \quad
\to {\rm H^d}_{p,\varPhi,\varPsi}(K, M){-\int_K\varPhi\,d\upsilon_g}
\ \ \text{as}\ \ \delta\to 0,
\end{aligned}
$$
but also

		$$
		\begin{aligned}
		&{\rm H^d}_{\Delta_p,\varPhi,\varPsi}(K, M){-2\int_K\varPhi\,d\upsilon_g}\\
		&\quad \leq  \int_{ M\setminus K}\Big(|\Delta_p w-\varPhi w^{p-1}|+\varPhi w ^{p-1}\Big)\,d\upsilon_g+\int_{\partial M}\varPsi w ^{p-1}d\sigma_g
		\\
	&\quad {\color{red}\leq \int_{ M\setminus K}\left(\frac{\big|(\dot{h}^{p-1})'(f_\ast)\big|}{|\nabla f_\ast|^{-p}}
+\Big|\frac{\varPhi}{(\dot{h}f_\ast)^{1-p}}-\frac{\varPhi}{ h^{1-p}(f_\ast)}\Big|+\frac{\varPhi}{w^{1-p}} \right)\,d\upsilon_g+\int_{\partial M}\frac{\varPsi}{w ^{1-p}}d\sigma_g}
 \\ &\quad  \to  {\rm H^d}_{p,\varPhi,\varPsi}(K, M){-\int_K\varPhi\,d\upsilon_g}+\int_{ M\setminus K}\varPhi f_\ast^{p-1} \,d\upsilon_g+\int_{\partial M}\varPsi f_\ast ^{p-1}d\sigma_g
\ \ \Big(\text{as}\ \delta\to 0\Big)
\\ &\quad = 2  {\rm H^d}_{p,\varPhi,\varPsi}(K, M){-2\int_K\varPhi\,d\upsilon_g} - \text{in other words} -
		\end{aligned}
$$
\begin{equation}
\label{e39e}{\rm H^d}_{\Delta_p,\varPhi,\varPsi}(K, M)\le 2  {\rm H^d}_{p,\varPhi,\varPsi}(K, M).
\end{equation}
		\end{itemize}
	
\item Thirdly, putting together \eqref{e31e}\&\eqref{e39e} immediately derives \eqref{e18} for any smooth condenser $(K, M)$.

\end{itemize}

\section{Addendum}\label{s3}

On the basis of \cite{CLW}, we can establish the following comparision law which is regarded as the first addition to Theorem \ref{t11}.

\begin{prop}
	\label{p32}
	Suppose that:
	\begin{itemize}
		\item
		
		$(M,g)$'s Ricci \& mean curvature pair $\{\mathsf{Ric},\mathsf{H}\}$ obeys
		\begin{equation*}
		\mathsf{Ric}\ge (n-1)\kappa\ \ \text{in}\
		\
		M\ \
		\& \ \
		\mathsf{H}\ge\lambda\	\ \text{on}\
		\
		\partial
		M;
		\end{equation*}
		
		\item
		$\big(\mathbb
		R^n(\kappa),g_\kappa\big)$
		is	the	space form with constant sectional curvature $\kappa$;
		\item
		$B_{\kappa,\lambda}$
		is	the geodesic ball in
		$\big(\mathbb R^n(\kappa),g_\kappa\big)$ with $\mathsf{H}=\lambda$ on $\partial B_{\kappa,\lambda}$;
		\item
		$$
		(M_{\kappa,\lambda},g_{\kappa,\lambda})=\begin{cases}
		\big(\bar{B}_{\kappa,\lambda},g_\kappa\big)&
		\text{as}\	\
		\kappa>0\
		\
		\text{or}\
		
		\kappa\le
		0\
		\&\
		\lambda>\sqrt{|\kappa|};\\
		\big(\mathbb
		R^n(\kappa)\setminus
		B_{\kappa,-\lambda},g_\kappa\big)&
		\text{as}
		\ \ \kappa\le0\ \ \&\ \
		\lambda<-\sqrt{|\kappa|};\\
		\big([0,\infty)\times\mathbb
		S^{n-1},
		dt^2+s^2_{\kappa,\lambda}(t)g_{\mathbb
			S^{n-1}}\big)&
		\text{as}\ \ \kappa\le0\ \ \&\ \
		|\lambda|=\sqrt{|\kappa|};\\
		\big([t_{\kappa,\lambda},\infty)\times\mathbb S^{n-1}, dt^2+s^2_{\kappa,0}(t)g_{\mathbb S^{n-1}}\big)&\text{as}\
		\ \kappa<0\ \ \&\ \
		|\lambda|<\sqrt{|\kappa|};
		\end{cases}
		$$
		\item
		$s_{\kappa,\lambda}(t)$ \& $t_{\kappa,\lambda}$
		are respectively the unique solutions of equations
		$$
		f''(t)+\kappa
		f(t)=0=f(0)-1=f'(0)+\lambda\ \ \&\ \
		\frac{s'_{\kappa,0}(t)}{s_{\kappa,0}(t)}=-\lambda;
		$$
		
		\item
		$\big\{\rho,\rho_{\kappa,\lambda}\big\}$
		is
		the
		distance
		function pair
		with
		respect
		to the pair
		$\big\{\partial
		M, \partial
		M_{\kappa,\lambda}\big\}$.
	\end{itemize}
	If
	\begin{equation}\label{e31a}
	\begin{cases}
	1<p<\infty;\\
	\delta=\underset{x\in
		K}{\min}\,\rho(x)>0;\\
	\sigma_g(\partial
	M)=\sigma_{g_{\kappa,\lambda}}({\partial}M_{\kappa,\lambda});\\
	\varPhi=\text{constant}\ge 0;\\
	\varPsi=\text{constant}\ge 0;\\
	K^*=\big\{x\in
	M_{\kappa,\lambda}:\
	\rho_{\kappa,\lambda}(x)\ge\delta\},
	\end{cases}
	\end{equation}
	then
	\begin{equation}\label{e32a}
	{\rm H^d}_{p,\varPhi,\varPsi}(K, M)\le	{\rm H^d}_{p,\varPhi,\varPsi}(K^*, M_{\kappa,\lambda}).
	\end{equation}
\end{prop}
\begin{proof}
	In order to	verifty \eqref{e32a}, let not only $u_*$ be
	the	minimizer of ${\rm H^d}_{p,\varPhi,\varPsi}(K^*, M_{\kappa,\lambda})$ -namely-
	$$
	{\rm H^d}_{p,\varPhi,\varPsi}(K^*, M_{\kappa,\lambda})=
	\int_{M_{\kappa,\lambda}}|\nabla
	u_*|^p\,d\upsilon_{g_{\kappa,\lambda}}+	\varPhi\int_{M_{\kappa,\lambda}}|u_*|^p\,d\upsilon_{g_{\kappa,\lambda}}+\varPsi\int_{\partial
		M_{\kappa,\lambda}}|u_*|^{p}\,d\sigma_{g_{\kappa,\lambda}}
	$$
	but also
	$$
	u(x)=\begin{cases}u_*(\rho(x))&\ \	\text{as}\ \ \rho(x)<\delta;\\
	1&\	\ \text{as}\ \ \rho(x)\ge\delta.
	\end{cases}
	$$
	Then not only there holds
	\begin{equation}\label{e37a}
	\begin{cases}
	-\Delta_p u_* +\varPhi|u_*|^{p-2}u_*=0&\text{in}\ \ M_{\kappa,\lambda}\setminus
	K^*;\\
	u_*=1&\text{in}\ \ K^*;\\
	|\nabla u_*|^{p-2}\nabla u_*\cdot{\bf n}+\varPsi|u_*|^{p-2}u_*=0&\text{on}\ \ \partial
	M_{\kappa,\lambda},
	\end{cases}
	\end{equation}
	but also the
	\eqref{e31a}-based proof
	of	\cite[Theorem
	1.3]{ChW}  derives
	
	\begin{align*}
	{\rm H^d}_{p,\varPhi,\varPsi}(K, M)&\le	\int_{M}|\nabla
	u|^p\,d\upsilon_g+\varPhi\int_{M}|u|^p\,d\upsilon_g+
	\varPsi\int_{\partial
		M}|u|^p\,d\sigma_{g}\\
	&\le \int_{M_{\kappa,\lambda}}|\nabla
	u_*|^p\,d\upsilon_g+ \varPhi\int_{M_{\kappa,\lambda}}|u_*|^p\,d\upsilon_{g_{\kappa,\lambda}}+\varPsi\int_{\partial
		M_{\kappa,\lambda}}|u_*|^{p}\,d\sigma_{g_{\kappa,\lambda}}
	\\
	&=
	{\rm H^d}_{p,\varPhi,\varPsi}(K^*, M_{\kappa,\lambda}),
	\end{align*}
	as desired in \eqref{e32a}.
\end{proof}

Here it is perhaps appropriate to make two comments on Proposition \ref{p32}.
	\begin{itemize}
		\item Upon letting not only $K=\emptyset$ (i.e., there is no conductor inside $M$) but also
		$\varPsi=\beta$ be a constant in $\mathbb R$ as well as $-\varPhi=\lambda_{p,\beta}$ be the first eigenvalue on the
		$p-$Laplacian with the Robin boundary condition (cf. \cite{AGM} for the Euclidean case) within \eqref{e37a}:
		\begin{equation}\label{e41}
		\begin{cases}
		-\Delta_p u_{\star}=\lambda_{p,\beta}|u_{\star}|^{p-2}u_{\star}&\text{in}\ \ M; \\
		|\nabla u_{\star}|^{p-2}{\nabla u_{\star} \cdot{\bf n}}+\beta|u_{\star}|^{p-2}u_{\star}=0 &  \text{on}\ \ \partial M,
		\end{cases}
		\end{equation}
		we read off that
		not only the so-called quasilinear Lapace-Robin eigenvalue $\lambda_{p,\beta}$ has the following characterization
		$$
		\lambda_{p,\beta}=\inf\limits_{u\in W^{1,p}(M)\setminus\{0\}}\frac{\int_M|\nabla u|^{p}\,d\upsilon_g+\beta\int_{\partial M}|u|^p\,d\sigma_g}{\int_M| u|^{p}\,d\upsilon_g},
		$$
		but also  the corresponding eigenfuntion $u_{\star}$ within \eqref{e41} is unique up to scaling, whence allowing us to take it as a positive function in $M$.
		
		\item Moreover, we can utilize \cite[Theorem 1.1]{CLW}'s second part for a given compact manifold $(M,g)$ to achieve that if there are not only $(M,g)$'s Ricci curvature condition
		\begin{equation}
		\label{e31}
		\mathsf{Ric}\ge (n-1)\kappa\in\Big\{n-1,0,1-n\Big\}
		\end{equation}
		but also
		\begin{equation}
		\label{e32}
		\begin{cases}
		(M_\kappa,g_\kappa)=\big(\mathbb S^n(R_\kappa),\tilde{g}\big);\\
		R_\kappa=\begin{cases} 1&\ \ \text{for}\ \ \kappa=1;\\
		\frac{\mathsf{d}}{\big(1+n\int_0^\pi\sin^{n-1}\theta\,d\theta\big)^\frac1n-1}&\ \ \text{for}\ \ \kappa=0;\\
		\frac1{c(\mathsf{d})}&\ \ \text{for}\ \ \kappa=-1;
		\end{cases}\\
		\mathsf{d}=\text{the diameter of $(M,g)$};\\
		c(\mathsf{d})=\text{the unique solution $\mathsf{u}$ to ${\mathsf{u}}\int_0^{\mathsf{d}}\big(\cosh t+{\mathsf{u}}\sinh t\big)^{n-1}dt=\int_0^\pi\sin^{n-1}\theta\,d\theta$};\\
		\alpha_\kappa=\frac{\upsilon_g(M)}{\upsilon_{g_\kappa}(M_\kappa)},
		\end{cases}
		\end{equation}
		plus
		\begin{equation*}
		\begin{cases}
		\text{$-\Delta_p u_\dagger =\lambda_p(\Omega)|u_\dagger|^{p-2}u_\dagger$ in $\Omega$};\\
		\text{$|\nabla u_\dagger|^{p-2}\nabla u_\dagger\cdot{\bf n}+|u_\dagger|^{p-2}u_\dagger=0$ on $\partial\Omega$};\\
		\Omega=\text{a smooth bounded domain in $(M,g)$};\\
		\Omega^\sharp=\text{the geodesic ball in $(M_\kappa,g_\kappa)$ with $\upsilon_g(\Omega)=\alpha_\kappa\upsilon_{g_\kappa}(\Omega^\sharp)$},
		\end{cases}
		\end{equation*}
		then there holds (cf. \cite{BFNT} for the Euclidean case)
		$$
		\lambda_p(\Omega)\ge\lambda_p(\Omega^\sharp).
		$$
		In parallel with the foregoing result, we can naturally obtain the following result that under not only \eqref{e31}-\eqref{e32} but also
		\begin{equation*}
		\begin{cases}
		\text{$-\Delta_p u_\ddagger = |u_\ddagger|^{p-2}u_\ddagger$ in $\Omega$};\\
		\text{$|\nabla u_\ddagger|^{p-2}\nabla u_\ddagger\cdot{\bf n}+\mu_p(\partial\Omega)|u_\ddagger|^{p-2}u_\ddagger=0$ on $\partial\Omega$};\\
		\Omega=\text{a smooth bounded domain in $(M,g)$};\\
		\Omega^\sharp=\text{the geodesic ball in $(M_\kappa,g_\kappa)$ with $\upsilon_g(\Omega)=\alpha_\kappa\upsilon_{g_\kappa}(\Omega^\sharp)$},
		\end{cases}
		\end{equation*}
		there exists
		$$
		\mu_p(\partial\Omega)\ge \mu_p(\partial\Omega^\sharp).
		$$
	
\end{itemize}

    Surprisingly yet naturally, we finger out the second addition to Theorem \ref{t11} as seen below.
    \begin{prop}\label{p31}
      Let $(M,g)$ be a smooth compact Riemannian $2\le n$-manifold with boundary $\partial M$.
      \begin{itemize}
      	\item [\rm (i)]
      If
      $$\begin{cases}
      (p,\beta)\in (1,\infty)\times\mathbb{R};\\
      \lambda_{p,\beta}\ \ \text{enjoys}\ \ \eqref{e41};\\
      \mathscr{F}_{p,\beta}(M)=\Big\{0\leq u\in C^{1,1}(M):\ \big(|\nabla u|^{p-2}\nabla u \cdot{\bf
      n}+\beta |u|^{p-2}u\big)\big|_{\partial M}=0\Big\};\\
      \Lambda_{p,\beta}=\inf\Bigg\{\frac{\int_M|\Delta_p u+\lambda_{p,\beta}u^{p-1}|\,d\upsilon_g+\beta\int_{\partial
      M}|u|^{p-2}u\,d\sigma_g}{\int_Mu^{p-1}\,d\upsilon_g}:
      u\in \mathscr{F}_{p,\beta}(M)\setminus\{0\} \Bigg\},
      \end{cases}
      $$
      then  there is the $(p,\beta)$-eigenvalue recycling law
    \begin{equation}\label{e41e}
    \Lambda_{p,\beta}=\lambda_{p,\beta}.
    \end{equation}

    \item[\rm (ii)]
    If
    $$\begin{cases}
    p\in (1,\infty);\\
    \mathscr{F}_{p,\infty}(M)=\Big\{0\leq u\in C^{1,1}(M): |\nabla u|\big|_{\partial M}= u\big|_{\partial M}=0\Big\};\\
    \lambda_{p,\infty}=\inf\limits_{u\in W^{1,p}_0(M)\setminus\{0\}}\frac{\int_M|\nabla u|^{p}\,d\upsilon_g}{\int_M| u|^{p}\,d\upsilon_g}=\text{
    	the quasilinear Laplace-Dirichlet eigenvalue};\\
    \Lambda_{p,\infty}=\inf\Bigg\{\frac{\int_M|\Delta_p u+\lambda_{p,\infty}u^{p-1}|\,d\upsilon_g}{\int_M u^{p-1}\,d\upsilon_g}:
    u\in \mathscr{F}_{p,\infty}(M)\setminus\{0\} \Bigg\},
    \end{cases}
    $$
    then there is the $(p,\infty)$-eigenvalue recycling law
    \begin{equation}\label{e42}
    \Lambda_{p,\infty}=\lambda_{p,\infty}.
    \end{equation}
\end{itemize}

\end{prop}
    \begin{proof} (i) The argument for \eqref{e41e} consists of four steps.
    	
      Firstly, for any $u\in \mathscr{F}_{p,\beta}(M)\setminus\{0\},$ we can assume that
      $$
      0\leq u\leq 1\ \ \text{in}\ \ M
      $$
      since the energy functional in the definition of  $\Lambda_{p,\beta}$ is a scaling invariant.

      Secondly, with the last assumption on $\{u,\lambda_{p,\beta}\}$ we can estimate
      $$\begin{aligned}
       &\int_M\big|\Delta_p u+\lambda_{p,\beta}u^{p-1}\big|\,d\upsilon_g+\beta\int_{\partial M}|u|^{p-2}u\,d\sigma_g
         \\ &\quad  \geq\int_M(1-2u)(\Delta_p u+\lambda_{p,\beta}u^{p-1})\,d\upsilon_g+\beta\int_{\partial M}|u|^{p-2}u\,d\sigma_g
       \\ &\quad =-\int_{M}|\nabla u|^{p-2}\nabla u\cdot\nabla (1-2u)\,d\upsilon_g+\int_{\partial M}(1-2u)|\nabla u|^{p-2}\nabla u\cdot{\bf
      n}\,d\sigma_g
        \\ &\quad \quad+\lambda_{p,\beta}\int_M(1-2u)u^{p-1}\,d\upsilon_g+\beta\int_{\partial M}|u|^{p-2}u\,d\sigma_g
      \\ &\quad =2\int_{M}|\nabla u|^p\,d\upsilon_g-\beta\int_{\partial M}(1-2u)|u|^{p-2}u\,d\sigma_g
      \\ &\quad \quad
      +\lambda_{p,\beta}\int_M(1-2u)u^{p-1}\,d\upsilon_g+\beta\int_{\partial M}|u|^{p-2}u\,d\sigma_g
      \\ &\quad =2\Bigg(\int_{M}|\nabla u|^p\,d\upsilon_g+\beta \int_{\partial M}  |u|^p\,d\sigma_g-\lambda_{p,\beta}\int_M u^{p}\,d\upsilon_g\Bigg)
       +\lambda_{p,\beta}\int_{M}u^{p-1}\,d\upsilon_g
      \\&\quad \geq  \lambda_{p,\beta}\int_{M}u^{p-1}\,d\upsilon_g,
      \end{aligned}
      $$
      thereby taking the infimum over the above assumed functions $u$ to  deduce
      \begin{equation}\label{e41ee}
      \Lambda_{p,\beta}\geq\lambda_{p,\beta}.
      \end{equation}

      Thridly, we choose $u_\ast$ to be the solution of \eqref{e41}, whence finding that $\Delta_p u_\ast$ is of $C^{1,\alpha}$ and so that $u_\ast$ can still be treated as a test function to compute
      $\Lambda_{p,\beta}$. Hence
      \begin{equation}\label{e41eee}
      \Lambda_{p,\beta}\leq \frac{\beta\int_{\partial
      M}|u_\ast|^{p-2}u_\ast\,d\sigma_g}{\int_Mu_\ast^{p-1}\,d\upsilon_g}=\lambda_{p,\beta}.
      \end{equation}

      Finally, a combination of \eqref{e41ee}-\eqref{e41eee} yields \eqref{e41e}.

      (ii) The argument for \eqref{e42} is almost the same as that for \eqref{e41e}. Therefore, a short verification is provided below.

      On the one hand, we can use the same method as in Proposition \ref{p31}  to obtain
      $$\Lambda_{p,\infty}\geq\lambda_{p,\infty}.$$

      On the other hand, if we consider the first eigenfunction $u_\ast$ to the Dirichlet problem
      $$
      \begin{cases}
       \Delta_p u=-\lambda_{p,\infty} u^{p-1}&\text{in}\ \ M;
       \\ u=0 & \text{on}\  \ \partial M,
      \end{cases}
      $$
      then $u_\ast$ is not suitable to calculate $\Lambda_{p,\infty}$ thanks to
      $$\nabla u_\ast\neq 0\ \ \text{on}\ \  \partial M.
      $$
      Instead, for any small $0<\delta\ll 1$ we can consider the revised function
      $$w=h(u_\ast)\ \ \text{with}\ \
      \ \begin{cases}
      h\in C^2[0,\infty);
      \\ h(0)=h'(0)=0;
      \\ h'(t)\geq0\ \ \forall\ t\geq0;
      \\ h(t)\big|_{[\delta,\infty)}=t.
      \end{cases}
      $$
      Then a similar analysis to \eqref{e23}-\eqref{eBe} gives
      $$\begin{aligned}
      &\int_M\big|\Delta_p w+\lambda_{p,\infty}w^{p-1}\big|\,d\upsilon_g
      \\ &\quad=\int_M\big|\dot{h}^{p-1}\Delta_p u_\ast +(\dot{h}^{p-1})'|\nabla u_\ast|^p+\lambda_{p,\infty}h^{p-1}(u_\ast)\big|\,d\upsilon_g
      \\ &\quad =\int_{\{0\leq u_\ast\leq\delta\}}\big|(\dot{h}^{p-1})'|\nabla u_\ast|^p+\lambda_{p,\infty}\big(h^{p-1}(u_\ast)-\dot{h}^{p-1}\big)u^{p-1}\big|\,d\upsilon_g
      \\ &\quad\leq \int_0^\delta|(\dot{h}^{p-1})'|\int_{\{u_\ast=t\}}|\nabla u_\ast|^{p-1}\,d\sigma_g dt+\mathcal{O}(\delta)
      \\ &\quad =\int_{\partial M}|\nabla u_\ast|^{p-1}\,d\sigma_g+\mathcal{O}(\delta)
      \\ &\quad =\lambda_{p,\infty} \int_{M} u_\ast ^{p-1}\,d\sigma_g+\mathcal{O}(\delta)
      \\ &\quad =\lambda_{p,\infty} \int_{M} w^{p-1}\,d\sigma_g+\mathcal{O}(\delta),
      \end{aligned}
      $$
      thereby not only reaching $$\Lambda_{p,\infty}\leq\lambda_{p,\infty}
      $$
      but also completing the argument for \eqref{e42}.
    \end{proof}

\bigskip
%\noindent{\bf Data availability}. Any dataset was neither generated nor analysed within this study.

%\noindent{\bf Conflict-of-interest statement}. The authors have no conflicts of interest to declare.

\end{document}